\documentclass[11pt]{amsart}
\usepackage{amsmath,amsthm,amsfonts,amssymb,amscd}
\usepackage[dvips]{graphicx}
\usepackage{psfrag}
\usepackage{a4wide}

\theoremstyle{plain}
\newtheorem{thm}{Theorem}[section]
\newtheorem{prop}[thm]{Proposition}
\newtheorem{lemma}[thm]{Lemma}
\newtheorem{clly}[thm]{Corollary}

\theoremstyle{definition}
\newtheorem{remark}[thm]{Remark}
\newtheorem{example}[thm]{Example}
\newtheorem{defi}[thm]{Definition}

\newcommand{\NN}{{\mathbb N}}
\newcommand{\ZZ}{{\mathbb Z}}

\renewcommand{\epsilon}{\varepsilon}

\newcommand{\diam}{\operatorname{diam}}

\newcommand{\Orb}{\operatorname{Orb}}

\def \NN {{\mathbb N}}

\def \ZZ {{\mathbb Z}}

\begin{document}

\title[ $CW$-expansiveness and specification for set-valued functions]
{Continuum-wise expansiveness and specification for set-valued functions and topological entropy}

\author{
Welington Cordeiro
and Maria Jos\'e Pac\'ifico 
 }
 \thanks{M.J.P. were partially supported by CNPq,
  PRONEX-Dyn.Syst., FAPERJ. W. C. were partially supported by CNPq.}

\maketitle

\begin{abstract}{We define the concept of continuum wise expansive for set-valued functions and prove that if a compact metric space admit a set-valued $cw$-expansive function then the topological entropy of $X$ is positive.} We also introduce the notion of pointwise specification property for set-valued functions and prove that set-valued functions with this property has positive entropy.
\end{abstract}


\section{Introduction}\label{sec:intro} 

The notion of expansiveness was introduced in the middle of the twentieth century in \cite{U}. Expansiveness is a property shared by a large class of dynamical systems exhibiting chaotic behavior. 
In \cite{Ma} Ma\~n\'e  proved that if a compact metric space
admits an expansive homeomorphism then its topological dimension
is finite. In \cite{Fa}  Fathi proved that if the topological dimension is greater than $1$ then the topological entropy of all expansive homeomorphism is positive. 
In the 90s, Kato introduced the notion of \emph{continuum-wise expansiveness} for homeomorphisms \cite{K1}, and extended both results 
(Ma\~n\'e and Fathi) to continuum-wise expansive homeomorphisms.

In \cite{BW} 
this concept is given for flows and it is proved that some properties valid for discrete dynamics are also valid for flows. 
Using this definition, Keynes and Sears \cite{KS} extend the results 
of Ma\~n\'e  for expansive flows. They proved that if a compact metric
space admits an expansive flow then its topological dimension is finite.
They also proved that expansive flows on manifolds with topological
dimension greater than $1$ have positive entropy.

In \cite{ACP} it was introduced the notion of \emph{positively continuum-wise expansiveness} for flows and the authors proved that $cw$-expansive  continuos flows on compact metric spaces with
topological dimension greater than $1$ have positive entropy, extending to continuos dynamics the result of Keynes and Sears.

In view of so many interesting results relating entropy to the chaotic dynamics of single-valued functions, it is natural to ask if the concept of entropy extends to set-valued functions and if so, if it gives nice informations about the dynamics of the set-valued function.

In \cite{Za} the author studied the asymptotic behavior of the trajectories of a set-valued dynamical system. In particular, he gived necessary and sufficient conditions for the existence of global attractors for dispersive systems as well as motivations coming from economic models.

Recently, in \cite{CMM}, the authors introduced two kinds of entropy, that they call separated and spanning entropy, 
 given by separated and spanning sets respectively. They prove that they keep a number of well-known properties of the topological entropy of single-valued maps.
 
Also recently, in  \cite{KT} the authors  introduced the notion of entropy that resemble  the classical notion of topological entropy of single-valued maps, and show that the entropy of a set-valued function is equal to the topological entropy of the shift map on its orbit spaces. Moreover, they give sufficient conditions for a set-valued function to have positive or infinite topological entropy. 

Finally, in \cite{RT}, the authors extend the notion of the specification property from the usual single-valued function setting to the setting of set-valued mappings and prove that specification implies topological mixing and positive entropy.

The aim of this paper is to introduce the notion of pointwise specification property and continuum-wise expansivity for set-valued functions. We extend \cite[Theorem 5.8]{K1} to this context and  give sufficient conditions to a continuum-wise expansive set-valued function have positive entropy. We compare separated and spanning entropy defined in \cite{CMM} with entropy defined in \cite{KT}. We also introduce the notion of point-wise specification for set-valued functions and prove that set-valued functions with this property have positive entropy and also positive separated entropy. 
We point out that there are set-valued functions satisfying the specification property second the notion given in \cite{RT} but do not satisfy the point-wise specification property, as shown in Example \ref{ex}. Moreover, there are set-valued functions satisfying the specification property which do not  have positive separated entropy, as shown in Example~ \ref{ex-vale}.

This paper is organized as follows. In Section \ref{sec:def} we introduce the notion of continuum-wise expansive set-valued functions. 
In Section \ref{s-topo} we give sufficient conditions to a set-valued function have positive entropy, generalizing \cite[Theorem 5.8]{K1}. In Section \ref{s-sepentro}
 we exploit the notion of separated and spanning entropy and compare these kind of entropy with the entropy of a set-valued function. Finally, in Section \ref{s-strong} we introduce the notion of point-wise specification for set-valued functions and prove that set-valued functions satisfying the point-wise specification property has positive entropy.

\section{ Continuum wise expansive set-valued functions}\label{sec:def}

Given a compact metric space $X$ with the metric $d$, let $2^X$ be the set of all non-empty subsets of $X$ with Hausdorff metric $d_H$, defined by 
$$d_H(A,B)=\max\{\sup_{x\in A}\inf_{y\in B}d(x,y),\sup_{y\in B}\inf_{x\in A}d(x,y)\},$$
if $A,B\in 2^X$. Given the subsets $A$ and $B$ of $X$, recall that the distance between $A$ and $B$ is given by
$$d(A,B)=\inf\{d(x,y);x\in A \text{ and }y\in B\}.$$
Let $K(X)\subset 2^X$ be the set of closeds subsets of $X$, and $C(X)\subset K(X)$ the set of compact and connected subsets of $X$. 

We say that a function $F:X\rightarrow Y$ is a \emph{set-valued map} if $Y\subset 2^X$.
If $f : X \to Y$ is a map, we may think of $f $ as a set-valued
function by defining a function $F_{f} : X \to 2^Y$ by 
$F_f(x) = \{f(x)\}$.

For increased distinction, we will refer to a function $F : X \to  2^Y$ as a set-valued function and a continuous function $f : X \to Y $ as a mapping.

  Let $X$ be a compact Hausdorff space and  $F:X\rightarrow 2^X$  a set-valued map 
  such that $F(x)$ is a compact set for each $x\in X$. Then $F$
 is  \emph{upper semi-continuous} at $x\in X$ if for each open set $V$ containing $F(x)$ there is an open set $U$ containing $x$ such that if $y\in U$ then $F(y)\subset V$. 

If $X,Y,$ and $Z$ are compact metric spaces, $F:X\to 2^Y$
 and $G:Y \to 2^Z$, we define $G\circ F:X\to 2^Z$ by
$$
G\circ F(x)=\bigcup_{y\in F(x)} G(y).
$$
Note that if $F$ and $G$ are upper semi-continuous, then $G\circ F$ is as well. 

If $F:X\rightarrow 2^X$ satisfies $F(x)$ is a non-empty compact set for each $x\in X$ then $(X,F)$ is called a {\em{topological dynamical system}}.
We define $F^0$ to be the identity on $X$, and for each $n\in\mathbb{N}$, $F^n=F\circ F^{n-1}$. 

Next we recall the concept  of orbit through $x \in X$ for the topological dynamical system $(X,F)$. 
For this, given $x\in X$, we start defining the sets


\[
\Sigma(x)=\{\overline{x}=(...,x_{-1},x=x_0,x_1,...); x_{i+1}\in F(x_{i})\}
\]
\[
\Sigma^+(x)=\{\overline{x}=(x=x_0,x_1,x_2,...); x_{i+1}\in F(x_i)\} \ \ 
\]
\[
\Sigma^-(x)=\{\overline{x}=(x=x_0,x_1,x_2,...); x_{i}\in F(x_{i+1})\} \ \ 
\]
\[\mbox{and for $n\in\mathbb{N}$} \,\,,
\Sigma_n(x)=\{(x=x_0,...,x_{n-1}); x_{i+1}\in F(x_{i})\}.
\]

\begin{defi}
Given $x\in X$, the full orbit of $x$ is defined as
$\text{{Orb}}(x,F)=\bigcup_{x\in X}\Sigma(x),$ the forward orbit of 
$x$  as \,\,$ \overrightarrow{Orb}(x,F)=\bigcup_{x\in X}\Sigma^+(x),$ and the backward orbit of $x$  as \,\,
$ \overleftarrow{Orb}(x,F)=\bigcup_{x\in X}\Sigma^-(x)$. For $n\in \NN$, the set $ \text{{Orb}}_n(x,F)=\bigcup_{x\in X}\Sigma_n(x)$ is a finite orbit at $x$.

A full orbit $\text{{Orb}}(x,F)$ is periodic if there is $n\in\NN$ such that $x_i=x_{i+n}$ for all $i\in\ZZ$. If $\text{{Orb}}(x,F)$ is periodic, the period of $x$ is the smallest number $n\in\NN$ satisfying $x_i=x_{i+n}$ for all $i\in\ZZ$.
\end{defi}

\vspace{0.2cm}

Given a set $A\subset X$ and $n\in\NN$, we define the following orbit spaces:
\[\Orb_n(A,F)=\cup_{x\in A}\Orb(x,F),\quad \overrightarrow{\Orb}(A,F)=\cup_{x\in A}\Orb(x,F)
\]
\[\overleftarrow{\Orb}(A,F)=\cup_{x\in A}\Orb(x,F)\quad
\Orb(A,F)=\cup_{x\in A}\Orb(x,F).
\]

We endow each of these spaces with the induced product topology.

If $\mathbb{A}\in\{\mathbb{Z},\mathbb{Z}_{\geq 0}\}$ we
consider the product space $\prod_{i\in\mathbb{A}}X$ with the metric  
$\rho$ defined as
\begin{equation}\label{metricarho}\rho(\overline{x},\overline{y})=\sup_{i\in\mathbb{A}}\frac{d(x_i,y_i)}{|i|+1}.\end{equation}
For $\prod_{i=1}^nX$ we define 
\begin{equation}\label{metricadn}
D_n(\overline{x},\overline{y})=\max_{0\leq i<n}d(x_i,y_i).
\end{equation}


If $A\subset X$ define $F^1(A)=F(A)=\bigcup_{x\in A}F(x)$ and $F^n(A)=F\circ F^{n-1}(A)$. 


By a \emph{continuum}, we mean a compact connected non degenerate metric space. 

Recall that a continuous function $f:X\rightarrow X$ is  \emph{positively continuum wise-expansive} if there is $\delta$ such that if $A\subset X$ is a continuous set then there is $n\in\mathbb{N}$ with $\diam(f^n(A))>\delta$. See \cite{K1}. Next we extend this concept to set-valued functions.

\begin{defi} A topological dynamical system $(X,F)$ is \emph{positively continuum-wise expansive} if satisfies:
\begin{enumerate}
\item $F(x)$ is a compact and connected for each $x\in X$;   
\item there is $\delta>0$ such that if $A\subset X$ is a continuum then there is $n\in\mathbb{N}$ such that $\diam F^n(A)>\delta$.
\end{enumerate}
\end{defi}
\begin{remark}Let $f:X\rightarrow X$ a continuous map with $X$ a compact metric space. If we define $F_f:X\rightarrow 2^X$ by $F_f(x)=\{f(x)\}$, then $(X,F_f)$ is a topological dynamical system. By definition, a continuous map $f:X\rightarrow X$ is  positively continuum-wise expansive map if, and only if, $(X,F_f)$ is positively continuum-wise expansive. Moreover, since $f$ is continuos, $(X,F_f)$ is upper semi-continuous.
\end{remark}

\noindent {\em{Notation}}\/ For short, we denote continuum-wise expansive by $cw$-expansive.

\section{Topological entropy for $cw$-expansive set-valued functions}\label{s-topo}

In this section we introduce the notion of topological entropy for $cw$-expansive set-valued functions
and prove that  $cw$-expansive set-valued function defined in a metric space with positive topological dimension has positive topological entropy.  

We start recalling  the notion of topological entropy for set-valued functions given in \cite{RT}.  For this, let $(X,F)$ be a topological dynamical system, $n\in\mathbb{N}$ and $\epsilon>0$. Given $ A\subset X$, we denote by $\#(A)$  the cardinality of 
 $A$.

A $(n,\epsilon)$-\emph{separated} set for $(X,F)$ is a $(n,\epsilon)$-separated subset of $\Orb_n(X,F)$. Define $$s_n(\epsilon)=\sup\{\#(E); E \text{ is an } (n,\epsilon)\text{-separated set}\}.$$ 
An $(n,\epsilon)$-\emph{spanning} set for $(X,F)$ is an $(n,\epsilon)$-spanning subset of $\Orb_n(X,F)$. Define 
$$r_n(\epsilon)=\inf\{\#(E); E \text{ is an } (n,\epsilon)\text{-spanning set}\}.$$

It is proved in \cite[Theorem 2]{RT} that
$$\lim_{\epsilon\rightarrow 0}\limsup_{n\rightarrow\infty}\frac{1}{n}\log r_n(\epsilon)=\lim_{\epsilon\rightarrow 0}\limsup_{n\rightarrow\infty}\frac{1}{n}\log s_n(\epsilon).$$
And the notion of entropy in \cite{RT} is the following:

\begin{defi}\label{d.entropia} The \emph{topological entropy} of $(X,F)$ is defined by
$$h(X,F)=\lim_{\epsilon\rightarrow 0}\limsup_{n\rightarrow\infty}\frac{1}{n}\log r_n(\epsilon).
$$
\end{defi}

Given a continuum $A\subset X$, recall that $C(A)$ is the set of all continuum subsets of $A$.
\begin{lemma}\label{l2} Let $A$ be a continuum of $X$ with $\diam A>c$. Then there are  $A_1,A_2\in C(A)$ such that $d_H(A_1,A_2)>\frac{c}{8}$ and
$$\diam(A_1)=\diam(A_2)=\frac{c}{8}.$$

\end{lemma}
\begin{proof} Let $a_1,b_2\in A$ be points with $d(a_1,a_2)>\frac{c}{2}$. It is well-known that there are arcs $\alpha_i:[0,1]\rightarrow C(X)$ from $\{a_i\}$ to $A$, such that if $t<s$ then $\alpha(t)\subset\alpha(s)$, $i=1,2$. Let $t_1,t_2\in[0,1]$ be the smallest times such that $\diam\alpha_i(t_i)=\frac{c}{8}$, $i=1,2$. Clearly, $A_1=\alpha(t_1)$ and $A_2=\alpha(A_2)$ satisfies $d_H(A_1,A_2)>\frac{c}{8}$.

\end{proof}

Next we give a technical lemma that we shall use in the proof of the main result in this section.

\begin{lemma}\label{1} Let $(X,F)$ be an upper semi-continuous topological dynamical system positively $cw$-expansive with constant $\eta>0$. For all $0<\delta<\eta$ and $\epsilon>0$ there is $N\in\mathbb{N}$ such that if $A$ is a continuum and $$\sup_{0\leq i\leq N}\diam(F^i(A))<\frac{\delta}{2},$$
then $\diam A<\epsilon$.
\end{lemma}
\begin{proof} Otherwise, there are $\epsilon>0$ and a sequence of continuum sets $A_n$ such that $\diam A_n\geq\epsilon$ and $$\sup_{0\leq i\leq n}\diam(F^i(A_n))<\frac{\delta}{2},\quad \forall \,\,
n\in\mathbb{N}.$$
 Since $C(X)$ is compact, we can assume that $A_n\rightarrow A$ with $A\subset X$ a continuum set. Here $C(X)$ is the set of compact and connected subsets of $X$.
  Hence $\diam(F^n(A))<\delta$ for all $n\in\mathbb{N}$. 
  In fact, fix $n\in\mathbb{N}$ and given $x\in X$ and $\epsilon >0$, denote by $B(x,\epsilon)$ the usual ball with center $x$ and radius $\epsilon$.
  Since $F$ is upper semi-continuous there is $\epsilon'>0$ such that for all $x\in X$ if $y\in B(x,\epsilon')$ then $F^n(y)\subset B(F^n(x),\frac{\delta}{4})$. However, there is $n_0>n$ such that $A\subset B(A_{n_0},\epsilon')$, therefore $F^n(A)\subset B(F^n(A_{n_0}),\frac{\delta}{4})$. Since $\diam F^n(A_{n_0})<\frac{\delta}{4}$, we have $\diam F^n(A)<\delta<\eta$. This is a contradiction and proves the lemma.

\end{proof}

If $X$ is a compact metric space, $\dim_{top}(X)$ means the topological dimension of $X$.

\begin{thm}\label{c1} Let $(X,F)$ be an upper semicontinuous $cw$-expansive topological dynamical system with constant of expansivity equals to $\eta>0$. If $\dim_{top}(X) >0$ then the topological entropy $h_{top}(X,F)$ is positive.
\end{thm}

\begin{proof} Let $A\subset C(X)$ be a non-trivial continuum set and $0<\delta<\min\{\eta,\diam A\}$. Let $N\in\mathbb{N}$ be given by Lemma \ref{1} relative to $\epsilon=\frac{\delta}{10}$. Then we can find a subcontinuum $A_0\subset A$ such that $\frac{\delta}{10}<\diam(A_0)<\delta$. Fix $m\in\mathbb{N}$. There is $0\leq l\leq N$ with $\diam(F^{i_1}(A))>\delta$. By Lemma \ref{l2} we can find  $A_1,A_2\subset F^l(A)$ continuum sets with $\diam(A_i)=\frac{\delta}{8}$, $i\in\{1,2\}$ and $d_h(A_1,A_2)>\frac{\delta}{8}$. Arguing in the same way we can find a finite collection $\{A_{i_1,i_2,...,i_m};i_k=1 \ \text{or }2 \}$ of continuum subsets of $X$ satisfying the following properties:
\begin{itemize}
\item $A_{i_1,...,i_j}$ is a continuum subset of $F^{l}(A_{i_1,...,i_{j-1}})$, $1\leq l=l(i_j-1)\leq N$, $i_k=1,2$;
\item $\diam A_{i_1,...,i_j}=\frac{\delta}{8}$, $i_k=1,2$, $1<j\leq m$;
\item $d(A_{i_1,...,i_j,1},A_{i_1,...,i_j,2})>\frac{\delta}{8}$, $i_k=1,2$, $1<j< m$. 
\end{itemize}
Then for each $(i_1,...,i_m)$ we can find an orbit $\overline{a}_{i_1,...,i_m}=(x_0,x_1,...)$ such that for each $j\in\{1,...,m\}$ it has $x_{l(i_1)+...+l(i_{j})}\in A_{i_1,...,i_j}$. Since $1\leq l(k)\leq N$ for each $k=1,...,m$, we have that the set $$E_m=\{\overline{a}_{i_1,...,i_m};i_j=1 \ \text{or }2\}$$ is a $(\frac{\delta}{8},Nm)$-separated set with $2^m$ elements. Thus $s_{Nm}(\frac{\delta}{8})\geq 2^m$ and         
\begin{eqnarray*} \limsup_{n\rightarrow\infty}\frac{1}{n}\log s_{n}(\frac{\delta}{8})&\geq& \limsup_{m\rightarrow\infty}\frac{1}{Nm}\log s_{Nm}(\frac{\delta}{8}) \\  
&\geq& \limsup_{m\rightarrow\infty}\frac{1}{Nm}\log s_{Nm}(\frac{\delta}{8}) \\
&\geq& \limsup_{m\rightarrow\infty}\frac{1}{Nm}\log 2^m = \frac{1}{N}\log 2>0.
\end{eqnarray*}   
Hence $h(X,F)>0$. The proof of Theorem \ref{c1} is complete.

\end{proof}
 

Combining Theorem \ref{c1} with \cite[Theorems 3.1 and 3.5]{RT} we obtain a generalization to the setting of set-valued functions of 
\cite[Theorem 5.8]{K1} :

\begin{clly} Let $(X,F)$ be an upper semicontinuous topological dynamical system $cw$-expansive with $\dim_{top}(X)>0$. Then
\begin{enumerate}
\item If $\overrightarrow{\sigma}:\ \overrightarrow{\Orb}(X,F)\rightarrow \ \overrightarrow{\Orb}(X,F)$ is the shift map defined as
$$\overrightarrow{\sigma}(x_0,x_1,...)=(x_1,x_2,...),$$
then $h(\overrightarrow{\sigma})>0$.
\item If $\overleftarrow{\sigma}: \ \overleftarrow{\Orb}(X,F)\rightarrow \ \overleftarrow{\Orb}(X,F)$ is the shift map defined as
$$\overleftarrow{\sigma}(...,x_{-1},x_0)=(...,x_{-2},x_{-1}),$$
then $h(\overleftarrow{\sigma})>0$.
\item If $\sigma:\Orb(X,F)\rightarrow \Orb(X,F)$ is the shift map defined as
$$\sigma(\{x_i\}_{i=-\infty}^{\infty})=(\{y_i\}_{i=-\infty}^{\infty}),$$
where $y_i=x_{i+1}$, then $h(\sigma)>0$.
\end{enumerate}
\end{clly}

\begin{example} Let $f:S^1\rightarrow S^1$ be a rotation of the circle. Define $F:S^1\rightarrow 2^{S^1}$ by $F(t)=[t,f(t)]$. Then $(S^1,F)$ is $cw$-expansive, but 
$\vec{\sigma}$
 and $\sigma$ are not $cw$-expansive maps. Indeed, there is $n\in\mathbb{N}$ such that for all point $t\in S^1$ we have $F(t)=S^1$. Therefore for each continuum $A\subset S^1$ we obtain that $F^n(A)=S^1$ implying that $(X,F)$ is $cw$-expansive. 
 On the other hand, for each $t\in S^1$ we have that
 $I^{\pm}(t)=(t,t,t,t,...)\in\Sigma^+(t)\cap\Sigma^{-}(t)$ and $I(t)=(...,t,t,t,t,...)\in\Sigma(t)$. 
 Given $\delta>0$ choose $t<t'$ such that $\diam[t_1,t_2]<\delta$. Then, $B=\{I(t);t\in[t_1,t_2]\}$ and $B^\pm=\{I^\pm(t);t\in[t_1,t_2]\}$ are continuum subsets of  \,
   $\overrightarrow{\Orb}(X,F)$. 
 
 But $\overrightarrow{\sigma}^n(B^\pm)=\overleftarrow{\sigma}^n(B^\pm)=B^\pm$ and $\sigma^n(B)=B$ for all $n\in \mathbb{N}$. 
 Hence $\overrightarrow{\sigma}$, $\overleftarrow{\sigma}$, and $\sigma$ are not positively $cw$-expansives maps.
\end{example}

If $A\subset X$ is a continuum, then $\pi(A)=\bigcup_{x\in A}\Sigma^+(x)$ is a subcontiuum of \ $\overrightarrow{\Orb}(X,F)$. However, if $A$ is a subcontinuum of  \ $\overrightarrow{\Orb}(X,F)$ then $\pi^{-1}(A)$ may be not a continuum in $X$. In fact, it can be only one point: in example \ref{ex} described below, for each $t\in S^1$, $\Sigma^+(t)$ is a continuum but $\pi^{-1}(\Sigma^+(t))=\{t\}$.  

\begin{prop} Let $(X,F)$ be a upper semicontinuous topological dynamical system and $\overrightarrow{\sigma}: \overrightarrow{\Orb}(X,F) \to \overrightarrow{\Orb}(X,F)$,
$\sigma: \Orb(X,F) \to \Orb(X,F)$ be the shift maps defined above.
Then
\begin{enumerate}
\item if $\overrightarrow{\sigma}$ is a positively $cw$-expansive map then $(X,F)$ is a positively $cw$-expansive topological dynamical system;
\item if $\sigma$ is a positively $cw$-expansive map then $(X,F)$ is a positively $cw$-expansive topological dynamical system.
\end{enumerate}
\end{prop}
\begin{proof}  $(1)$ Let $\delta>0$ be a $cw$-expansive constant to $\overrightarrow{\sigma}$.
 If $A\subset X$ is a continuum set, then $\pi(A)=\bigcup_{x\in A}\Sigma^+(x)$ is a subcontinuum of $\overrightarrow{Orb}(X,F)$.
  By the choice of $\delta$, there is $n_1\in\mathbb{N}$ such that 
  $\diam(\overrightarrow{\sigma}^{n_1}(\pi(A))>\delta$. 
 Since $\pi(A)$ is a compact set, we can find $\overline{x},\overline{y}\in \pi(A)$ such that $\rho(\overrightarrow{\sigma}^{n_1}(\overline{x}),\overrightarrow{\sigma}^{n_1}(\overline{y})>\delta$. 
 By definition of $\rho$ at (\ref{metricarho}), there is $n_2\in\mathbb{N}$ such that $$\frac{d(x_{n_1+n_2},y_{n_1+n_2})}{n_2+1}>\delta.$$
Since $x_{n_1+n_2},y_{n_1+n_2}\in F^{n_1+n_2}(A)$, we get $\diam(F^{n_1+n_2}(A))>(n_2+1)\delta\geq\delta$, proving $(1)$.

(2) The prove is similar to the previous item and we leave to the reader.

\end{proof}

\section{Separated and spanning topological entropy }\label{s-sepentro}

In this section we exploit the notion of separated and spanning entropy introduced in \cite{CMM}. For this,  let $X$ be a compact metric space and $F:X\to 2^{X}$ be a set-valued function defined on $X$. Given $n \in \NN$ we consider the product space $\prod_{i=1}^nX$  endowed with the metric $D_n$ defined at
(\ref{metricadn}).

For $n\in\mathbb{N}$ define the map $d_n:X\times X\rightarrow\mathbb{R}$ by
$$d_n(x,y)=\inf\{D_n(\overline{x},\overline{y});\overline{x}\in\Sigma_n(x) \text{ and }\overline{y}\in\Sigma_n(y) \}.$$
These maps are metrics in the single-valued case but in general they are only semi-metrics, see \cite{Do,Gi} and \cite[Example 4.5]{CMM}.

The $\epsilon$-ball centered at $x\in X$ and radius $\epsilon$ with respect to $d_n$ is given by
$$B_n[x,\epsilon]=\{y \in X;d_n(x,y)\leq\epsilon\}.$$
For $\epsilon>0$ and $n\in\mathbb{N}$ we say that $E\subset X$ is an $(d_n,\epsilon)$ separated  set if 
$$B_n[x,\epsilon]\cap E=\{x\}, \ \ \forall x\in E.$$
We let $S_n(\epsilon)$ be the largest cardinality of a $(d_n,\epsilon)$ separated set and define the $se$-separated entropy of $(X,F)$ by
$$h_{se}(X,F)=\lim_{\epsilon\rightarrow 0}\limsup_{n\rightarrow\infty}\frac{1}{n}\log S_n(\epsilon).$$

For $\epsilon>0$ and $n\in\mathbb{N}$ we say that $G\subset X$ is an $(d_n,\epsilon)$-spanning set if 
$$X=\bigcup_{x\in G} B_n[x,\epsilon].$$
We let $R_n(\epsilon)$ be the smallest cardinality of a $(d_n,\epsilon)$-spanning set and define the spanning entropy of $X,F)$ by
$$h_{sp}(X,F)=\lim_{\epsilon\rightarrow 0}\limsup_{n\rightarrow\infty}\frac{1}{n}\log R_n(\epsilon).$$

It is proved in \cite[Theorem 3.5]{CMM} that $h_{sp}(X,F)\leq h_{se}(X,F)$ and that they coincide if and if $d_n$ is a metric \cite[Theorem 3.6]{CMM}.
 Next, we compare the separated entropy to the entropy of a topological dynamical system $(X,F)$ given in Definition \ref{d.entropia}.

\begin{thm}\label{hseh} Let $(X,F)$ be a topological dynamical system. Then $h_{se}(X,F)\leq h(X,F)$.  
\end{thm}
\begin{proof}  Let $n\in\mathbb{N}$ and $\epsilon>0$ be given. Let $E$ be a $(d_n,\epsilon)$-separated set with maximal cardinality. For each $x\in E$ choose $\overline{x}\in\Sigma(x)$. Let $G\subset \Orb(E,F)$ be the set of all these sequences and 
$$H=\{\overline{x(n)};\overline{x}\in G\}.$$ 
If $\overline{x(n)},\overline{y(n)}\in H$, since $x\in E$ we get $B_n[x,\epsilon]\cap E=\{x\}$, and as also $y\in E$ we get    
$$\inf\{d_n(a,b);a\in\Sigma_n(x) \text{ and } b\in\Sigma_n(y)\}>\epsilon.$$
Since $\overline{x(n)}\in\Sigma_n(x)$ and $\overline{y(n)}\in\Sigma_n(y)$, we get $d_n(x(n),y(n))>\epsilon$. Hence $H$ is an $\epsilon$-separated subset of $\Orb_n(X,F)$.

Thus, since we proved that $S_n(\epsilon)\leq s_n(\epsilon)$ for all $n\in\mathbb{N}$ and $\epsilon>0$, we conclude that $h_{se}(X,F)\leq h(X,F)$.

\end{proof}

\begin{example} Here we give an example of a dynamical system $(X,F)$ such that $h_{se}<h(X,F)$. 
Let $f:S^1\rightarrow S^1$ defined by $f(t)=t^2$. Define a topological dynamical system $(S^1,F)$ by $F(t)=B[f(t),\delta]$ for some $\delta>0$. So there is $n\in\mathbb{N}$ (depending only $\delta$) such that $F^n(t)=S^1$ for all $t\in S^1$. Therefore, $(S^1,F)$ is a $cw$-expansive dynamical system and since $\dim S^1=1$  Theorem \ref{c1} implies  that $h(S^1,F)>0$. On the other hand by \cite[Example 4.9]{CMM} implies that $h_{se}=0$. 
\end{example}

\section{Point-wise specification property for set-valued functions}\label{s-strong}
In this section we introduce the notion of point-wise specification property for set-valued functions and prove that set-valued functions satisfying this property have positive entropy.
For this, we start given the definition of specification introduced in \cite{RT}.

\begin{defi}A topological dynamical systems $(X,F)$ has the \emph{(periodic) specification property} if for each $\epsilon>0$ there is $M>0$ such that, for any $x^1,...,x^n\in X$, any $a_1\leq b_1<...<a_n\leq b_n$ with $a_{i+1}-b_i>M$, and any orbits $(x_j^i)_{j=0}^\infty$, (and for any $P>M+b_n-a_1$,) there exists a point $z\in X$ that has an orbit $(z_j)_{j=0}^\infty$ such that $d(z_j,x_j^i)<\epsilon$, for $i\in\{1,...,n\}$ and $a_i\leq j\leq b_i$ (and $z_P=z$).
\end{defi}

\begin{remark} The notion  of periodic specification property was called specification property in the \cite{RT}. Note that periodic specification property implies specification property. 
\end{remark}

The following result was proved for periodic specification property in \cite{RT}. The proof for specification property is essentially the same, and we leave it to the reader.

\begin{thm}\cite[theorem 5]{RT} Let $(X,F)$ be a topological dynamical system 
such that $\#(X) \geq 2.$ 
If $(X,F)$ has the specification property then $h(X,F)>0$.
\end{thm}
Next we define point-wise specification property for set-valued functions.

\begin{defi}A topological dynamical system $(X,F)$ has the \emph{point-wise specification property} if for each $\epsilon>0$ there is $M>0$ such that, for any $x^1,...,x^n\in X$, any $a_1<a_2<...<a_n$ with $a_{i+1}-a_i>M$, there exists a point $z\in X$ such that $d_H(F^j(z),\{x^i\})<\epsilon$, for $i\in\{1,...,n\}$.
\end{defi}

\begin{remark} Let $f:X\rightarrow X$ be a continuous function on a compact metric space $X$. If $f$ has the specification property, then $(X,F_f)$ has the pointwise specification property.  Here
$F_{f} : X \to 2^X$ is the set-valued function induced by $f$, defined by 
$F_f(x) = \{f(x)\}$. 
\end{remark}     


\begin{example}\label{ex}  We will show a simple topological dynamical system with the periodic specification property which does not have the pointwise specification property. Let $X$ be a compact and connected metric space with $\diam(X)>1$ and let $F:X\rightarrow 2^{X}$ be defined by $F(x)=X$, for each $x\in X$. 
On one hand,  let $\epsilon>0$ be given and take $M=0$, $x^1,...,x^n\in X$,  $a_1\leq b_1<...<a_n\leq b_n$ with $a_{i+1}-b_i>0$, orbits $(x_j^i)_{j=0}^\infty$, and $P>1+b_n-a_1$. 
If we choose $z=x^1_{a_1}$ since $F(x)=X$ we can find an orbit $(z_j)_{j=0}^\infty$ of $z$ such that $z_j=x_j^i$, for $i\in\{1,...,n\}$ and $a_i\leq j\leq b_i$ and $z_P=z$. On the other hand, assume that $(X,F)$ has the strong specification property. Take $M\geq 0$ for $\epsilon=\frac{1}{2}$. Choose $x\in X$, and take $a_1=1$. For any $z\in X$, since $F(z)=X$ we have $$d_H(F^{a_1}(z),\{x\})=d_H(X,\{x\})>\frac{1}{2}=\epsilon,$$
leading to a contradiction. Therefore, $(X,F)$ does not have the point-wise specification property.      
\end{example}

\begin{example}\label{ex-vale} Periodic specification property do not implies $h_{se}(X,F)>0$.
Indeed, applying \cite[Theorem 5]{RT}  to  the previous example  we get\,\,$h(X,F)>0$. Since $F^n(x)=X$ for all $n\geq 1$, $d_n(x,y)=d(x,y)$. Hence $S_n(\epsilon)=C_\epsilon$, for all $n\in\mathbb{N}$. Therefore,
$$\limsup_{n\rightarrow\infty}\frac{1}{n}\log S_n(\epsilon)=\limsup_{n\rightarrow\infty}\frac{1}{n}\log C_\epsilon=0 \quad \Rightarrow \quad h_{se}(X,F)=0.$$
\end{example}

\begin{defi} A topological dynamical systems $(X,F)$ is \emph{topologically mixing} if, for any non empty subsets $U$ and $V$ in $X$, there is $M\in\mathbb{N}$ such that for each $m>M$ there is $x_0\in U$ with  orbit $(x_j)_{j=0}^\infty$ such that $x_m\in V$. See \cite{RT}.
\end{defi}

\begin{thm} Let $(X,F)$ be a topological dynamical system. If $(X,F)$ has the point-wise specification property then $(X,F)$ is topologically mixing.
\end{thm}

\begin{proof}  Let $U$ and $V$ be open sets in $X$, take $x\in U$, $y\in V$ and $\epsilon>0$ sucht that $B(x,\epsilon)\subset U$ and $B(y,\epsilon)\subset V$. Let $M$ be given by point-wise specification property. By point-wise specification property, for each $m\in\mathbb{N}$ there is $z^m\in X$ such that $d_H(\{z^m\},\{x\})<\epsilon$ and $d_H(F^{M+m}(z^m),\{y\})<\epsilon$. Take $(z^m_n)_{n=0}^\infty\in\Sigma(z^m)$. Since $z^m_0=z^m$ and $z^m_{M+m}\in F^{M+m}(z^m)$,
$$d(z^m_0,x)=d(z^m,x)=d_H(\{z^m\},\{x\})<\epsilon \text{ and }d(z^m_{M+m},y)\leq d_H(F^{M+m}(z^m),\{y\})<\epsilon.$$
Hence $z^m_0\in U$ and $z^m_{M+m}\in V$. Then $(X,F)$ is topologically mixing.

\end{proof}

\begin{thm}\label{SSPE} Let $(X,F)$ be a topological dynamical system  with cardinality of $X$ greater than $2$. If $(X,F)$ has the point-wise specification property then $h_{se}(X,F)$ is positive.
\end{thm}
\begin{proof}  Take $x\neq y\in X$, $0<\epsilon<\frac{d(x,y)}{3}$ and $M$ given by point-wise specification property for this $\epsilon$. For $n\in\mathbb{N}$, define 
$$A_n=\{(c^1,...,c^n); c^i=x,y\}.$$
Then $A_n$ has $2^n$ elements. For each $c=(c^1,...,c^n)\in A_n$, if $a_i=(i-1)M$ there is a point $z^c\in X$ such that $$d_H(F^{(i-1)M}(z^c),\{c^i\})<\epsilon.$$
Hence, if $c\neq d\in A_n$, there is $i_0\in\{1,...,n\}$ such that $c^{i_0}\neq d^{i_0}$. Therefore, if $\overline{z_c}\in\Sigma(z^c)$ and $\overline{z^d}\in\Sigma(z^d)$ 
\begin{eqnarray}\label{D}D_{nM}(\overline{z}^c,\overline{z}^d)=\sup_{0\leq i\leq (n-1)M}d(z_i^c,z_i^d)\geq d(z_{(i_0-1)M}^c,z_{(i_0-1)M}^d).
\end{eqnarray}
Since $z_{(i_0-1)M}^c\in F^{(i_0-1)M}(z^c)$ and $z_{(i_0-1)M}^d\in F^{(i_0-1)M}(z^d)$,
$$d(z_{(i_0-1)M}^c,c^{i_0})\leq d_H(F^{(i_0-1)M}(z^c),\{c^{i_0}\})<\epsilon \text{ and } d(z_{(i_0-1)M}^d,d^{i_0})\leq d_H(F^{(i_0-1)M}(z^d),\{d^{i_0}\})<\epsilon.$$
Thus 
$$d(z_{(i_0-1)M}^c,z_{(i_0-1)M}^d)\geq d(c^{i_0},d^{i_0})-d(z_{(i_0-1)M}^c,c^{i_0})-d(z_{(i_0-1)M}^d,d^{i_0})> 3\epsilon-\epsilon-\epsilon=\epsilon.$$
By \ref{D} $D_{nM}(\overline{z}^c,\overline{z}^d)>\epsilon$. Hence if $$E_n=\{z_c;c\in A_n\}$$ we have that cardinality of $E_n$ is $2^n$ and for each $z,w\in E_n$ we have $d_n(z,w)>\epsilon$, ie, $B_n[z,\epsilon]\cap E_n=\{z\}$ for each $z\in E_n$. Therefore, $E_n$ is an $(d_n,\epsilon)$-separated set. Hence for a fixed $\epsilon>0$ we get
\begin{eqnarray*}\limsup_{n\rightarrow\infty}\frac{1}{n}\log(S_n(\epsilon))&\geq& \limsup_{n\rightarrow\infty}\frac{1}{nM}\log(S_{nM}(\epsilon)) \\ &\geq&\limsup_{n\rightarrow\infty}\frac{1}{nM}\log(2^n) \\ &=& \limsup_{n\rightarrow\infty}\frac{1}{M}\log(2)=\frac{1}{M}\log(2).
\end{eqnarray*} 
Then $h_{se}(X,F)$ is positive and we conclude the proof.

\end{proof}

Combining the theorems \ref{SSPE} and \ref{hseh} we have,

\begin{clly}\label{SSPh} Let $(X,F)$ be a topological dynamical with cardinality of $X$ greater than $2$. If $(X,F)$ has the point-wise specification property, then $h(X,F)$ is positive.
\end{clly}

Combining the corollary \ref{SSPh} and \cite[Theorem 3.1, Theorem 3.5]{KT} we have,

\begin{clly} Let $(X,F)$ be a topological dynamical system with cardinality of $X$ greater than $2$. If $(X,F)$ has the point-wise specification property, then $h(\stackrel{\leftarrow}{\sigma})>0$, $h(\stackrel{\rightarrow}{\sigma})>0$ and $h(\sigma)>0$.
\end{clly}

\begin{example} A non-trivial topological dynamical system with point-wise specification property. Let $f:\mathbb{T}^2\rightarrow\mathbb{T}^2$ the classical Anosov linear on the torus. Define $F:\mathbb{T}^2\rightarrow 2^{\mathbb{T}^2}$ by $F(x)=\{f(x)\}$ if $x$ is a periodic point for $f$, and $F(x)=\mathbb{T}^2$ if $x$ is not a periodic point for $f$. Since $f$ is an Anosov transitive homeomorfism there is the periodic specification property. Let $\epsilon>0$ be given and $M>0$ be given by periodic specification property for $f$. Let $x^1,...,x^n\in \mathbb{T}^2$ and $a_1<a_2<...<a_n$ with $a_{i+1}-a_i>M$ be given. For each $1\leq i\leq n$ take $y^i=f^{-a_i}(x^i)$. By choice of $M$ there is $z\in\mathbb{T}^2$, a periodic point of $f$, such that $d(f^{a_i}(z),f^{a_i}(y^i))<\epsilon$, for all $1\leq i\leq n$. Hence for each $1\leq i\leq n$,
\begin{eqnarray*} d_H(F^{a_i}(z),\{x^i\})=d_H(\{f^{a_i}(z)\},\{f^{a_i}(y^i)\})=d(f^{a_i}(z),f^{a_i}(y^i))<\epsilon.
\end{eqnarray*} 
Therefore $(X,F)$ has the point-wise specification property.      
\end{example}

\vspace{1cm}
\noindent
{\em W. Cordeiro}\/
 Mathematics Department,
The Pennsylvania State University,
State College, PA 16802, USA. 
\noindent E-mail:  wud11@psu.edu 
\vspace{0.2cm}

\noindent{\em M. J. Pacifico }\/
\noindent Instituto de Matem\'atica,
Universidade Federal do Rio de Janeiro,
C. P. 68.530\\ CEP 21.945-970,
Rio de Janeiro, RJ, Brazil.
\noindent E-mail:   pacifico@im.ufrj.br

\end{document}